\documentclass[final,12pt,a4paper]{amsart}
\usepackage{graphicx}
\usepackage[all]{xy}
\usepackage{placeins}
\usepackage{enumitem}
\usepackage{amssymb}
\usepackage{latexsym}
\usepackage{amsmath}
\usepackage{mathrsfs}
\usepackage{array,booktabs}

\usepackage[notref, notcite]{showkeys}  % for draft versions only
\usepackage{xcolor}

\usepackage{hyperref}
\hypersetup{
    colorlinks,
    linkcolor={red!50!black},
    citecolor={blue!50!black},
    urlcolor={blue!80!black}
}
\newcommand{\harxiv}[1]{\href{http://arxiv.org/abs/#1}{\texttt{arXiv:#1}}}
\newcommand{\hyref}[2]{\hyperref[#2]{#1~\ref*{#2}}}

\newcommand{\coloneqq}{\mathrel{\mathop:}=}

\newcommand{\sleq}{\mathrel{(\leq)}}

\newcommand{\orthstack}[1]{\mathrel{ \substack{ \resizebox{!}{1.2ex}{$\bot$} \\[-0.2ex]
                                           \resizebox{!}{0.8ex}{$#1$} }}}
\newcommand{\ndorth}{ \orthstack{-\Delta} }

\newcommand{\dorth}{ \orthstack{\Delta} }

\usepackage{fullpage}

\theoremstyle{plain}
\newtheorem{theorem}{Theorem}[section]

\newtheorem{lemma}[theorem]{Lemma}
\newtheorem{corollary}[theorem]{Corollary}
\newtheorem{proposition}[theorem]{Proposition}

\newtheorem{introtheorem}{Theorem}

\newtheorem{introproposition}[introtheorem]{Proposition}

\theoremstyle{definition}
\newtheorem{remark}[theorem]{Remark}

\newtheorem*{naive-algorithm}{Na\"ive algorithm}
\newtheorem*{refined-algorithm}{Refined algorithm}
\newtheorem*{definition}{Definition}

\newcommand{\IQ}{{\mathbb{Q}}}

\newcommand{\IN}{{\mathbb{N}}}

\newcommand{\ko}{{\mathcal O}}

\newcommand{\kp}{{\mathcal P}}

\newcommand{\rank}{\mathrm{rk}}

\newcommand{\dual}{^\vee}

\newcommand{\shat}[1]{\smash{\hat{#1}}}

\DeclareMathOperator{\Pic}{\mathsf{Pic}}

\DeclareMathOperator{\Coh}{\mathsf{Coh}}
\DeclareMathOperator{\Db}{\mathsf{D}^b}

\newcommand{\FM}{\mathsf{FM}}

\newcommand{\Fp}{\mathsf{F}_{\!\!+}}
\newcommand{\Fm}{\mathsf{F}_{\!\!-}}
\newcommand{\LL}{\mathbf{L}}
\newcommand{\RR}{\mathbf{R}}

\newcommand{\pr}{\mathsf{pr}}

\DeclareMathAlphabet{\mathpzc}{OT1}{pzc}{m}{it}

\newcommand{\cM}{\mathscr{M}}

\newcommand{\cU}{\mathscr{U}}

\newcommand{\into}{\hookrightarrow}
\newcommand{\onto}{\twoheadrightarrow}

\renewcommand{\iff}{\Longleftrightarrow}

\renewcommand{\phi}{\varphi}
\renewcommand{\epsilon}{\varepsilon}

\newcommand{\bib}[6]{{\bibitem{#2} #3: {\emph{#4},} #5#6.}}

\newcommand{\bibno}[1]{}

\begin{document}

\title{Stability of Picard sheaves \\ for vector bundles on curves}

\author{Georg Hein}
\author{David Ploog}
%\date{\today}
%% \newcount\myhours
%% \newcount\myminutes
%% \myhours   = \time
%% \divide   \myhours by 60
%% \myminutes = \time
%% \multiply \myhours by 60
%% \advance  \myminutes by -\myhours
%% \divide   \myhours by 60
%% \date{\today, \the\myhours:\ifnum\myminutes<10 0\fi \the\myminutes}

\begin{abstract}
We show that for any stable sheaf $E$ of slope $\mu(E) > 2g-1$ on a smooth, projective curve of genus $g$,
the associated Picard sheaf $\shat{E}$ on the Picard variety of the curve is stable. 
We introduce a homological tool for testing semistability of Picard sheaves.
\end{abstract}

\begingroup
  \renewcommand\thefootnote{}  %Key words: Picard sheaf, stable vector bundle, Jacobian variety.}%
  \addtocounter{footnote}{-1}%
\endgroup

\maketitle

%\tableofcontents

\addtocontents{toc}{\protect{\setcounter{tocdepth}{-1}}}  % No toc entry for Introduction

\section*{Introduction}
\addtocontents{toc}{\protect{\setcounter{tocdepth}{1}}}   % but enable toc entries for other sections

\noindent
Throughout, $X$ is a smooth, projective genus $g$ curve over an algebraically closed field $k$.
Let $\Pic \coloneqq \Pic^0(X)$ be the Picard variety of $X$ and $\kp$ the Poincar\'e line bundle on $X\times \Pic$.

For a vector bundle $E\in\Coh(X)$, its \emph{Picard complex} is the Fourier--Mukai (or integral) transform
$\shat{E} \coloneqq \FM_\kp(E)$, an object of $\Db(\Pic)$.
We denote its two cohomology sheaves by $\shat{E}^0$ and $\shat{E}^1$ and call these the \emph{Picard sheaves} of $E$.
Our goal is to show that $\shat{E}$ is (semi)stable on $\Pic$ for general, (semi)stable
bundles $E$ on $X$ for certain slopes. In fact, we prove this by showing that $\shat{E}$ is semistable when
restricted to curves $i\colon X \into \Pic$.
We have the following result; see Corollaries \ref{cor:res1}, \ref{cor:STAB1}, \ref{cor:res2}, and \ref{cor:STAB2}:

\begin{introtheorem} \label{thm:intro1}
If $E$ is a stable bundle on $X$ of slope $\mu(E) > 2g-1$, then the Picard sheaf $\shat{E}^0$ is stable on $\Pic$.
Dually, if $E$ is stable of slope $\mu(E) < -1$, then the Picard sheaf $\shat{E}^1$ is stable.
The analogous statements hold for semistability, using the non-strict inequalities.
\end{introtheorem}

Using our concept of orthogonality, we obtain results for Picard sheaves for
generic semistable bundles of slope $\mu\in[g-2,g]$ unless $\mu=g-1$; see \hyref{Proposition}{prop:STAB3} and \hyref{Corollary}{cor:STAB4}.

\begin{introtheorem}
For $\mu \in (g-1,g]$, there exists a semistable bundle $E$ on $X$ of slope $\mu$ such that its Picard sheaf $\shat{E}^0$ is semistable. Dually, for $\mu \in [g-2,g-1)$, there exists a semistable bundle $E$ on $X$ of slope $\mu$ such that its Picard sheaf $\shat{E}^1$ is semistable.
\end{introtheorem}

In order to show \hyref{Theorem}{thm:intro1}, we generalise Clifford's theorem about estimating global sections, from divisors to not necessarily semistable vector bundles. If $E = L_1\oplus\cdots\oplus L_r$ is a direct sum of line bundles with all $\deg(L_i) \in [0,2g-2]$, then $h^0(L_i) - 1 \leq \deg(L_i)/2$ by the classical Clifford theorem. This sums up to $h^0(E) - r \leq \deg(E)/2$. Therefore, the best generalisation one can hope for is the following result; see \hyref{Proposition}{Clifford}, where we  also give precise information about the equality case.

\begin{introproposition}
Let $E$ be a vector bundle of rank $r$ and degree $d$ on the smooth projective curve $X$ of genus $g$.
If $\mu_{\max}(E) \leq 2g-2 $ and $\mu_{\min}(E) \geq 0$, then we have   the estimate
\[ h^0(E)-r \leq \frac{d}{2} \, .\]
%Moreover, if $\mu_{\max}(E) < 2g-2 $ and $\mu_{\min}(E) > 0$ and $h^0(E)-r = \frac{d}{2}$, then $X$ is hyperelliptic, and the determinant line bundle $\det(E)$ is a multiple of the hyperelliptic line bundle $M$, and $E$ possesses a filtration with graded object ${\rm gr}(E)= \bigoplus_{i=1}^r M^{ \otimes a_i}$ with $0<a_i<g-1$.
\end{introproposition}

The special case of semistable vector bundles of slope $\mu \in [0,2g-2]$ was already proved in \cite[Theorem 2.1]{BGN}. If the slope of a semistable bundle $E$ is not in this interval, then either $H^0(E)=0$ or $H^1(E)=0$, and the dimension of the remaining cohomology group is computed by the Riemann--Roch theorem.

\subsection*{Known results}

\noindent
Classically, Picard sheaves are pushforwards of Poincar\'e bundles along the projection $X\times\Pic^d(X) \to \Pic^d(X)$, see \cite{Mattuck, Schwarzenberger, Mukai}. Later on, this notion was extended to pushforwards of universal bundles onto the moduli space.

Kempf \cite{Kempf} has shown that the Picard bundle on $\Pic^d(X)$ is stable for $d=2g-1$, the smallest degree where the Picard complex $\shat{\ko}_X = (\shat{\ko}_X)^0$ is a vector bundle. Here and later on, stability on the Picard variety is meant with respect to the polarisation by the theta divisor.
In \cite{EL}, Ein and Lazarsfeld proved the stability of Picard bundles on $\Pic^d(X)$ for $d\geq 2g$. They do this by restricting the Picard bundles to the canonical curves $X\subset\Pic^d(X)$ and $(-X)\subset\Pic^d(X)$.

Li \cite{Li} considers the moduli space $\cM_{r,d}$ of stable bundles on $X$ of rank $r$ and degree $d$ with $d > 2r(g-1)$ and $(d,r)=1$. If $\cU$ denotes the universal bundle on $X\times\cM_{r,d}$, then Li shows that the Picard bundle $\pr_{2*}\cU$ on $\cM_{r,d}$ is stable if $d>2gr$.

In \cite{BBGN}, the authors consider the same question for the moduli space of stable bundles of rank $r$ and fixed determinant $L$. Then the associated Picard bundle on the moduli space is stable (with respect to the unique polarisation) if $\deg(L) > 2r(g-1)$ and $(r,\deg(L))=1$.

Here, we go back to the case of Jacobians, but now we consider the Poincar\'e bundle twisted by a vector bundle pulled back from the curve. In other words, we study preservation of (semi)stability for the Fourier--Mukai transform along the Poincar\'e bundle, taking bundles on the curve to sheaves on its Picard variety.

\subsubsection*{Acknowledgments}
This work has been supported by SFB/TR 45
``Periods, moduli spaces and arithmetic of algebraic varieties''.

\subsubsection*{Conventions}
$X$ is always a smooth, projective curve of genus $g$ over a fixed algebraically closed field.
The \emph{slope} of a vector bundle is denoted $\mu(E) = \deg(E)/\rank(E)$.
Note $\mu(E\otimes F) = \mu(E) + \mu(F)$.
We repeatedly use the Riemann--Roch formula, always in the form $\chi(E) = \rank(E)(\mu(E) + 1-g)$.
We write $h^0(E) = \dim H^0(E)$ and $h^1(E) = \dim H^1(E)$.
We denote projections by $\pr_X,\pr_Y\colon X\times Y\to X$ or by $\pr_1,\pr_2\colon X\times X\to X$.
Given sheaves $E$ and $F$ on $X$, then as usual we write $E\boxtimes F = \pr_1^*E\otimes\pr_2^*F$.
Sometimes, we follow standard usage and pack two statements into one, using (semi)stability and $\sleq$.

\section{Orthogonality and stability}

\subsection{Definition of orthogonality and first properties}

\noindent
We first recall a classical result of Faltings \cite[Theorem~1.2]{Faltings}, which expresses semistability of a vector bundle on a curve as an orthogonality condition on $X$:

\begin{theorem}[Faltings 1993] \label{thm:Faltings}
A vector bundle $E$ on $X$ is semistable if and only if there exists a vector bundle $0\ne F$ such that $H^*(E\otimes F)=0$.
\end{theorem}

\begin{proof}
For the convenience of reader, we prove the easy implication of this statement. Assume $H^*(E\otimes F)=0$, and let $0\ne G\subset E$ be a destabilising subsheaf, i.e.\ $\mu(G) > \mu(E)$.

Then $G\otimes F \subset E\otimes F$ with $\mu(G\otimes F) > \mu(E\otimes F)$.
Since $H^*(E\otimes F)=0$, Riemann--Roch gives $\chi(G\otimes F) > \chi(E\otimes F) = 0$. Hence $h^0(G\otimes F) > 0$, which contradicts
 $H^0(G\otimes F) \subseteq H^0(E\otimes F) = 0$.
\end{proof}

Here, we introduce two other orthogonality notions, on $X\times X$, which work well with Picard sheaves:

\begin{definition}
Two coherent sheaves $E$ and $F$ on $X$ are called \emph{orthogonal with respect to $-\Delta$} if the sheaf
 $\pr_1^*E \otimes \ko_{X\times X}(-\Delta) \otimes \pr_2^* F$ on $X\times X$ has vanishing cohomology. Analogously, $E$ and $F$ are \emph{orthogonal with respect to $\Delta$} if the cohomology groups of $\pr_1^*E \otimes \ko_{X\times X}(\Delta) \otimes \pr_2^* F$ vanish. In short, we write
\begin{align*}
   E \ndorth F &\iff H^*(E \boxtimes F(-\Delta)) = 0 , \\
   E \dorth F  &\iff H^*(E \boxtimes F(\Delta))  = 0 .
\end{align*}
\end{definition}

Orthogonality has the following simple numerical description:

\begin{lemma} \label{lem:numORTH}
Let $E$ and $F$ be coherent sheaves on $X$.
\begin{enumerate}[label=(\roman*)]
  \item If $H^*(E\otimes F)=0$, then $\mu(F) = -\mu(E)+g-1$. \\
  \item If $E\ndorth F$, then $\mu(F) = g + \displaystyle\frac{g}{\mu(E)-g}$.
\end{enumerate}
\end{lemma}

\begin{proof}
(i) follows readily from Riemann--Roch.

(ii) The cohomology of the exact sequence $0 \to E\boxtimes F(-\Delta) \to E\boxtimes F \to E\boxtimes F|_\Delta \to 0$, using
 $E\boxtimes F|_\Delta \cong E\otimes F$ and $H^*(E\boxtimes F(-\Delta))=0$ from $E\ndorth F$ gives:
 $H^*(E\otimes F) \cong H^*(E\boxtimes F) \cong H^*(E) \otimes H^*(F)$, the latter isomorphism from the K\"unneth formula. Hence
 $\chi(E\otimes F) = \chi(E)\chi(F)$, or
 $(\mu(E)+\mu(F)+1-g) = (\mu(E)+1-g)(\mu(F)+1-g)$ by Riemann--Roch.
Manipulating this equation yields the claimed formula.
\end{proof}

We collect the following statements for referability; the proofs are immediate:

\begin{lemma} \label{lem:sym}
 For vector bundles $E$ and $F$ on $X$, we have the three equivalences:
\begin{flalign*}
  &\qquad\text{(i)}    & E \ndorth F &\iff F \ndorth E, & \\[0.5ex]
  &\qquad\text{(ii)}   & E \dorth F  &\iff F \dorth E,  & \\[0.5ex]
  &\qquad\text{(iii)}  & E \ndorth F &\iff (\omega_X \otimes E \dual ) \dorth (\omega_X \otimes F \dual)
                        \qquad \text{(Serre duality). } &
\end{flalign*}
\end{lemma}

\begin{definition}
We define two functors $\Fp, \Fm \colon \Coh(X) \to \Db(X)$ by
\begin{align*}
  \Fp(E) &= \RR\pr_{2*}(\ko_{X \times X}(\Delta) \otimes \pr_1^*E) , \\
  \Fm(E) &= \RR\pr_{2*}(\ko_{X \times X}(-\Delta) \otimes \pr_1^*E) .
\intertext{We denote the cohomology sheaves by}
  \Fp^i &= R^i\pr_{2*}(\ko_{X \times X}(\Delta) \otimes \pr_1^*E)
\end{align*}
and similarly for $\Fm^i(E)$. Since the fibres are 1-dimensional and objects in $\Db(X)$ decompose into their cohomology sheaves, we have
  $\Fp(E) = \Fp^0(E) \oplus \Fp^1(E)[-1]$, and
  $\Fm(E) = \Fm^0(E) \oplus \Fm^1(E)[-1]$.
\end{definition}

\begin{lemma}\label{lem:ortho1}
For vector bundles $E$ and $F$ on $X$, we have the two equivalences
\[\begin{array}{lrcll}
       (i) &  E \ndorth F  & \iff & H^*(\Fm^0(E) \otimes F) = 0 \text{, and } H^*(\Fm^1(E) \otimes F) = 0
    & \text{ and} \\[0.5ex]
      (ii) &  E \dorth F   & \iff & H^*(\Fp^0(E) \otimes F) = 0 \text{, and } H^*(\Fp^1(E) \otimes F) = 0 .
\end{array}\]
\end{lemma}

\begin{proof}
We only show (i), as the proof of (ii) works analogously.

Let
 $G=\pr_1^*E \otimes \ko_{X \times X}(-\Delta) \otimes \pr_2^* F$.
We compute the cohomology of $G$
% $\pr_1^*E \otimes \ko_{X \times X}(-\Delta) \otimes \pr_2^* F$,
using the Leray spectral sequence for $\pr_2$. For dimension reasons, the spectral sequence degenerates, thus
\[  H^0(G)= H^0(\pr_{2*}G) , \qquad  H^2(G)= H^1(R^1\pr_{2*}G) \,, \]
and a short exact sequence
\[ 0 \to H^1(\pr_{2*}G) \to H^1(G) \to  H^0(R^1\pr_{2*}G) \to 0 \,.\]
Therefore, we get $R^i\pr_{2*}(G) = \Fm^i(E) \otimes F$, using the projection formula, together with the definitions of $G$ and $\Fm^i(E)$. This proves both implications of the assertion, noting that $E \ndorth F$ is tantamount to $H^*(G)=0$.
\end{proof}

\begin{proposition} \label{prop:delta-orthogonality}
For a coherent sheaf $E$ on $X$, the following conditions are equivalent:
\begin{enumerate}[itemindent=1em,labelsep=1em]
    \item There exists a coherent $F \ne 0$ such that $E \ndorth F$.
    \item  $\Fm^0(E)$ and $\Fm^1(E)$ are semistable sheaves of the same slope.
\end{enumerate}
Similarly there exists such an equivalence for orthogonality with respect  to $+\Delta$.
\begin{enumerate}[itemindent=1em,labelsep=1em]
    \item[(1')] There exists a coherent $F \ne 0$ such that $E \dorth F$.
    \item[(2')]  $\Fp^0(E)$ and $\Fp^1(E)$ are semistable sheaves of the same slope.
\end{enumerate}
\end{proposition}

\begin{proof} We start with $(1) \implies (2)$. Assume $E \ndorth F$ for some $F \ne 0$.
From \hyref{Lemma}{lem:ortho1} we conclude
 $H^*(\Fm^i(E) \otimes  F) = 0$ for $i \in \{0,1\}$.
By the easy direction of \hyref{Theorem}{thm:Faltings}, $\Fm^i(E)$ is semistable.
Moreover, we get $\mu(\Fm^i(E)) = -\mu(F) + g - 1$ from \hyref{Lemma}{lem:numORTH}(i)

$(2) \implies (1)$. With $\Fm^0(E)$ and $\Fm^1(E)$ semistable of the same slope, their direct sum $\Fm^0(E) \oplus \Fm^1(E)$ is semistable as well. Thus by \hyref{Faltings' Theorem}{thm:Faltings}, there exists a sheaf $F \ne 0$  such that
 $H^*((\Fm^0(E) \oplus \Fm^1(E) )\otimes F)=0$.
By \hyref{Lemma}{lem:ortho1} we are done.
\end{proof}

\begin{remark}
Whenever we have an orthogonal pair $E \ndorth F$ with non-zero sheaves $E$ and $F$, then we conclude that the six sheaves
 $E$, $\Fm^0(E)$, $\Fm^1(E)$, $F$, $\Fm^0(F)$, and $\Fm^1(F)$
are semistable.

However, in most situations we consider here, one of the two sheaves $\Fm^0(E)$ or $\Fm^1(E)$ will be zero.
Anyway, they cannot be both zero, as the following argument shows:
Set $r=\rank(E)$, $R=\rank(\Fm^0(E))-\rank(\Fm^1(E))$, $d=\deg(E)$, and $D=\deg(\Fm^0(E))-\deg(\Fm^1(E))$.
A short Riemann--Roch computation along the lines of \hyref{Lemma}{lem:numORTH}(ii) gives
\[
    \left( \begin{array}{c} R\\D\end{array} \right) =
    \left( \begin{array}{rr} -g & 1\\0 & -1\end{array} \right)
    \left( \begin{array}{c} r\\d\end{array} \right)  .
\]
Thus, we can deduce the pair $(r,d)$  from $(R,D)$ unless $g=0$.
\end{remark}

\subsection{Picard sheaves, and embedding $X$ into the Picard scheme}
We denote by $\Pic = \Pic^0(X)$ the Picard scheme of the smooth curve $X$ and by $\kp$ the Poincar\'e bundle on $X\times\Pic$. Fixing a point $P_0 \in X(k)$, we normalise the Poincar\'e bundle by the additional assumption that
 $\kp|_{\{P_0\} \times \Pic} \cong \ko_{\Pic}$.
The projections from $X\times\Pic$ will be denoted
\[ \xymatrix{ X & X \times \Pic \ar[l]_(0.57){\pr_X} \ar[r]^-{\pr_P} & \Pic} . \]
For a coherent sheaf $E$ on $X$, we define its \emph{Picard complex} to be the object
\[ \shat{E} \coloneqq \FM_\kp(E) = \RR\pr_{P*}(\kp \otimes \pr_X^* E) \]
in the derived category $\Db(\Pic)$.
Since $\kp$ is $\pr_1$-flat, and $\pr_2$ is of dimension 1, we have only two
cohomology sheaves of our complex, and we call these the \emph{Picard sheaves} of $E$:
\begin{align*}
 \shat{E}^0 &\coloneqq \FM^0_\kp(E) =    \pr_{P*}(\kp \otimes \pr_X^* E)  \, , \quad \mbox{ and } \\
 \shat{E}^1 &\coloneqq \FM^1_\kp(E) = R^1\pr_{P*}(\kp \otimes \pr_X^* E) .
\end{align*}
We are interested mainly in the case when one of $\shat{E}^0$ or $\shat{E}^1$ is zero.
To study their semistability, we will restrict them to the curves $(X)_M$ and
$(-X)_N$, to be defined next.

For any line bundle $M$ of degree 1, we have an embedding of
 $\iota_M\colon X \to \Pic$ given by $P \mapsto M(-P)$. The image of $\iota_M$ is denoted by
$(-X)_M$. In the same way, any line bundle $N$ of degree $-1$ defines an embedding
 $\iota_N \colon X \to \Pic$ by $\iota_N(P)=N(P)$ with image $(X)_N$.
The next proposition gives the restriction of
$\shat{E}^i$ to the curves $(X)_N$ and $(-X)_M$.

\begin{proposition} \label{prop:FM-restriction}
Let $E$ be a coherent sheaf on $X$. \\
For arbitrary line bundles $N\in\Pic^{-1}(X)$ and $M\in\Pic^1(X)$, there are isomorphisms
\begin{align*}
 \FM_\kp(E) \otimes^\LL \ko_{(X)_N}  &\cong \Fp(E \otimes N) \otimes \ko_X(-P_0) , \\
 \FM_\kp(E) \otimes^\LL \ko_{(-X)_M} &\cong \Fm(E \otimes M) \otimes \ko_X(P_0)
\end{align*}
in the derived category of $(X)_N \cong X \cong (-X)_M$.
Moreover, for general line bundles $N\in\Pic^{-1}(X)$ and $M\in\Pic^1(X)$, there are isomorphisms of sheaves
\[ \begin{array}{l@{~}c@{~}l @{\qquad} l@{~}c@{~}l}
      \shat{E}^0|_{(X)_N}   &\cong& \Fp^0(E \otimes N) \otimes \ko_X(-P_0) , &
      \shat{E}^1|_{(X)_N}   &\cong& \Fp^1(E \otimes N) \otimes \ko_X(-P_0) \\[0.5ex]
      \shat{E}^0|_{(-X)_M}  &\cong& \Fm^0(E \otimes M) \otimes \ko_X(P_0) , &
      \shat{E}^1|_{(-X)_M}  &\cong& \Fm^1(E \otimes M) \otimes \ko_X(P_0) .
  \end{array}\]
\end{proposition}

\begin{proof}
For the formula for $\FM_\kp(E) \otimes^\LL \ko_{(X)_N}$, we compute
\begin{align*}
         \FM_\kp(E) \otimes^\LL \ko_{(X)_N} 
  &=     \RR\pr_{P*}( \pr_X^*E  \otimes \kp) \otimes^\LL \ko_{(X)_N}      & \text{(definition of $\FM_\kp$)} \\
  &\cong \RR\pr_{P*}( \pr_X^*E  \otimes \kp \otimes \pr_P^* \ko_{(X)_N})  & \text{(projection formula)} \\
  &\cong \RR\pr_{P*}( \pr_X^*E  \otimes \kp|_{ X \times (X)_N})            & \text{($\kp$ locally free)} \\
  &\cong \RR\pr_{2*}( \pr_1^*(E \otimes N) \otimes \ko_{X \times X}(\Delta) \otimes \pr_2^*\ko_X(-P_0))
             \hspace{-1em}                                             & \text{($\star$)} \\
  &\cong \RR\pr_{2*}( \pr_1^*(E \otimes N) \otimes \ko_{X \times X}(\Delta) ) \otimes \ko_X(-P_0)
                                                                       & \text{(projection formula)} \\ 
  &=     \Fp(E \otimes N) \otimes \ko_X(-P_0)                          & \text{(definition of $\Fp$)} 
\end{align*}
where in $\text{($\star$)}$, we identify $(X)_N$ with $X$, so the universal family $\kp$ restricted to  $X \times (X)_N$ is the line bundle
 $\pr_1^*N \otimes \ko_{X \times X}(\Delta) \otimes \pr_2^*\ko_X(-P_0)$ on $X \times X$.

Choosing a nice resolution $0 \to E_1 \to E_0 \to E \to 0$ where the $E_i$ are vector bundles with $\mu_{\max }(E_i) < 0$, we see that $\FM_\kp(E)$ can be represented by the complex of vector bundles $\shat{E}_1^1 \to \shat{E}_0^1$ on $\Pic$. For a general $N$, the curve $(X)_N$ does not contain any of the associated components of the cohomologies of that complex. Thus tensoring with $\ko_{(X)_N}$ is exact. Similarly for $M$.
\end{proof}

\subsection{Orthogonality proofs of semistability of Picard sheaves}

\begin{proposition} \label{prop:generic-Picard-semistability}
For a coherent sheaf $E$ on $X$, the following implications hold
\[ (1) \iff (2) \implies (3) \impliedby (2') \iff (1') \]
among these five conditions:
\begin{enumerate}
\item[(1)] $(E \otimes M) \ndorth F$ for some line bundle $M$ of degree 1,
           and a sheaf $F \ne 0$.
\item[(2)] For a general line bundle $M$ of degree 1, the restrictions of the
           Picard sheaves $\shat{E}^0$ and $\shat{E}^1$ to the curve $(-X)_M$ are
           both semistable of the same slope.
\item[(3)] The sheaves $\shat{E}^0$ and $\shat{E}^1$ are $\mu$-semistable with
           respect to the theta divisor on $\Pic$.
\item[(1')] $(E \otimes N) \dorth F$ for some line bundle $N$ of degree $-1$,
            and a sheaf $F \ne 0$.
\item[(2')] For a general line bundle $N$ of degree $-1$, the restrictions of
            the Picard sheaves $\shat{E}^0$ and $\shat{E}^1$ to the curve $(X)_N$
            are both semistable of the same slope.
\end{enumerate}
\end{proposition}

\begin{proof}
The proof of the equivalence $(1) \iff (2)$ follows from Propositions
\ref{prop:delta-orthogonality} and \ref{prop:FM-restriction}, having in mind that twisting with the line bundle $\ko_X(P_0)$ does not affect semistability. The implication $(2) \implies (3)$ is standard: given a sheaf $F$ on a variety $Y$ such that $F|_C$ is semistable, where $C\subset Y$ is a curve cut out by divisors in the linear system $|H|$ of the polarisation $H$, then $F$ is $\mu_H$-semistable on $Y$ --- a destabilising subsheaf $F'\subset F$ would induce a destabilising subsheaf $F'|_C \subset F|_C$.
\end{proof}

\begin{remark}
Indeed, the implications $(2) \implies (3) \Longleftarrow (2')$ were also
used in \cite{EL} as a main tool. Using orthogonality we can not detect
whether a semistable sheaf is stable.
\end{remark}

\subsection{A first example}

\begin{lemma} \label{lem:orth-pair}
Let $E$ be a semistable coherent sheaf on $X$ of degree 0.
Then there exists a coherent sheaf $F \ne 0$ such that $E \ndorth F$.
\end{lemma}

\begin{proof}
By \hyref{Faltings' Theorem}{thm:Faltings}, there exists a vector bundle $F$ with $H^*(E \otimes F)=0$. Since $F$ can be taken in an open subset in the moduli space of rank $R$ and degree $R(g-1)$ vector bundles, we may furthermore assume that  $H^*(F)=0$.
Tensoring the short exact ideal sheaf sequence of $\Delta$ on $X \times X$ with $E \boxtimes F=\pr_1^*E \otimes \pr_2^*F$, we obtain
\[ 0 \to E \boxtimes F(-\Delta) \to E \boxtimes F \to E\boxtimes F|_\Delta \to 0 .\]
We have $H^*(E \boxtimes F) = H^*(E) \otimes H^*(F)=0$ by the K\"unneth formula, and $H^*(E\boxtimes F|_\Delta) = H^*(E \otimes F) =0$ by our assumption. Hence $H^*(E \boxtimes F(-\Delta))=0$, i.e.\ $E\ndorth F$.
\end{proof}

\begin{corollary} \label{cor:res1}
Let $E$ be a semistable vector bundle on $X$ with $\mu(E)=-1$.
Then $\shat{E}^0 = 0$ and $\shat{E}^1$ is a vector bundle of rank $g \cdot \rank(E)$ which is also semistable. 
Moreover, the restriction of $\shat{E}^1$ to any curve $(-X)_M$ is semistable.
\end{corollary}

\begin{proof}
Since $E$ is semistable of negative degree, we have $\shat{E}^0 =0$, and it follows that $\shat{E}^1$ is a vector bundle of the given rank.
Now for any line bundle $M$ of degree 1, the tensor product $E \otimes M$ is semistable of degree 0.
Thus by \hyref{Lemma}{lem:orth-pair}, there exists a sheaf $F\neq0$ such that $(E \otimes M) \ndorth F$, and \hyref{Proposition}{prop:generic-Picard-semistability} shows the semistability of $\shat{E}^1$.

Finally, $\shat{E}^1$ is semistable when restricted to any curve $(-X)_M$ because the base change argument from the proof of \hyref{Proposition}{prop:FM-restriction} simplifies, with $\FM_\kp(E) = \shat{E}^1[-1]$ a shifted vector bundle.
\end{proof}

\subsection{A more subtle example}

\begin{lemma}\label{orth-pair2}
Assume that the genus $g \ne 1$.
Let $E$ be a vector bundle on $X$ of slope $\mu(E)=g-1$. Then the following three conditions are equivalent:
\begin{enumerate}[itemindent=1em,labelsep=1em]
\item $H^*(E)=0$.
\item $E \ndorth \ko_X$.
\item There exists a coherent sheaf $F \ne 0$ such that $E \ndorth F$.
\end{enumerate}
\end{lemma}

\begin{proof}
$(1) \implies (2)$:
From $H^*(E)=0$ we get $H^*(E\boxtimes\ko_X) = H^*(E) \otimes H^*(\ko_X) = 0$ and $H^*(E\otimes\ko_X|_\Delta) = H^*(E) = 0$.
Therefore, the long exact cohomology sequence of $0\to E\boxtimes\ko_X(-\Delta) \to E\boxtimes\ko_X \to E\otimes\ko_X|_\Delta \to 0$ gives $E\boxtimes\ko_X(-\Delta)) = 0$, i.e.\ $E\ndorth \ko_X$.

$(2) \implies (3)$: This implication is obvious.

$(3) \implies (1)$:
If we have such an orthogonality, then $E$ and $F$ are semistable. Note that $\deg(F)=0$ by \hyref{Lemma}{lem:numORTH}(ii).
Moreover, $E$ is also orthogonal to the direct sums $F^{\oplus N}$ for all $N>0$, which are semistable sheaves of rank $N \cdot \rank(F)$ and degree 0.
We may take $N$ big enough such that the theta divisor $\Theta_E$ in the moduli space of semistable degree 0 bundles on $X$ of rank $N \cdot \rank(F)$ is effective. By semicontinuity, being orthogonal to $E$ is an open condition. Thus, we may assume there exists an $F'$ outside $\Theta_E$, i.e.\ $E \ndorth F'$ and $H^*(E \otimes F') = 0$.
Again we consider the long exact cohomology sequence of $0\to E\boxtimes F'(-\Delta) \to E\boxtimes F' \to E\otimes F'|_\Delta \to 0$; here it yields $H^*(E) \otimes H^*(F') =0$.
We have $\chi(F') \ne 0$ from Riemann--Roch, if $g\ne 1$. So $H^*(F') \ne 0$, hence $H^*(E)=0$.
\end{proof}

\begin{corollary}
Let $E$ be a semistable sheaf of slope $g-2$.
If there exists a line bundle $M$ of degree 1 with $H^*(E \otimes M)=0$, then $\shat{E}^0=0$, and $\shat{E}^1$ is semistable of rank $\rank(E)$.
\end{corollary}

\begin{remark}
Raynaud proved in \cite{Raynaud} the existence of stable sheaves $E$ having integral slope with the following property: $H^*(E \otimes M) \ne 0$ for all line bundles $M$. These base points of the theta divisor form a proper closed subset of the moduli space. Thus, we can only hope that the Picard sheaves $\shat{E}$ are semistable for general semistable sheaves $E$.
\end{remark}

\section{A classical proof for the stability of $\shat{E}$}

\noindent
We take $E$ to be a globally generated vector bundle on $X$. Thus, we have
\[ 0 \to \Fm^0( E) \to H^0(E) \otimes \ko_X \to E \to 0     \qquad  \text{ and } \qquad
   \Fm^1(E) \cong H^1(E) \otimes \ko_X . \]
Note that a semistable bundle of slope $>2g-1$ is globally generated.

For a subsheaf $E' \subset E$ we obtain an injective map $H^0(E') \into H^0(E)$.
From both morphisms we eventually obtain a subsheaf $\Fm^0(E') \subset \Fm^0(E)$.
The next result tells us, that these subsheaves are enough to test
(semi)stability.

\begin{lemma}\label{lem:good1}
The sheaf $\Fm^0(E)$ is (semi)stable if for all globally generated subsheaves $E' \subset E$ we have the
inequality $\mu(\Fm^0(E')) \sleq \mu(\Fm^0(E) )$.
\end{lemma}

\begin{proof}
We have the short exact sequence
 $0 \to \Fm^0(E) \to H^0(E)\otimes\ko_X \to E \to 0$,
as $E$ is globally generated.
Let $U\subset\Fm^0(E)$ be a subbundle. The inclusion $U\into H^0(E)\otimes\ko_X$ induces a surjection
 $H^0(E)\dual\otimes\ko_X \onto H^0(U\dual)\otimes\ko_X$. Denote by $V\dual\subseteq H^0(U\dual)$ the image of the induced map on global sections, $H^0(E)\dual\to H^0(U\dual)$. The commutative diagram
\[ \xymatrix{
   H^0(E)\dual \otimes \ko_X \ar@{->>}[d] \ar@{->>}[rrd] \\
       V\dual \otimes \ko_X  \ar@{^{(}->}[r] & H^0(U\dual) \otimes \ko_X \ar[r] & U\dual
} \]
shows that $V\dual\otimes\ko_X\to U\dual$ is surjective. We get inclusions
 $U\into V\otimes\ko_X\into H^0(E)\otimes\ko_X$, which combine into a commutative diagram of short exact sequences
\[ \xymatrix{
   0 \ar[r] & U  \ar[d] \ar[r] & V\otimes\ko_X \ar[d] \ar[r] & E' \ar[r] \ar[d] & 0 \\
   0 \ar[r] & \Fm^0(E) \ar[r] & H^0(E)\otimes\ko_X   \ar[r]  & E  \ar[r]        & 0
} \]
where all vertical arrows are injective. Denote $v=\dim(V)$ and $r' = \rank(E')$. Then
\[ \mu(U) = -\frac{\deg(E')}{v-r'} \leq -\frac{\deg(E')}{h^0(E')-r'} = \mu(\Fm^0(E')) , \]
which shows that testing the (semi)stability condition only on subsheaves of the form $\Fm^0(E')$ suffices to deduce it for arbitrary subsheaves.
\end{proof}

\begin{corollary}\label{cor:general1}
If $E$ is a (semi)stable vector bundle of slope $\mu(E) > 2g$, then $\Fm^0(E)$ is (semi)stable.
\end{corollary}

\begin{proof}
Since $\mu \coloneqq \mu(E) > 2g-2$, we have $h^1(E)=0$ and $h^0(E)=\deg(E)+(1-g)\rank(E)$.
From the short exact sequence
 $0 \to \Fm^0(E) \to H^0(E) \otimes \ko_X \to E \to 0$,
we deduce
\[ \mu(\Fm^0(E)) = \frac{-\deg(E)}{h^0(E)-\rank(E)} = \frac{-\deg(E)}{\chi(E)-\rank(E)}
                 = \frac{-\deg(E)}{\deg(E)-g\cdot\rank(E)} = \frac{-\mu}{\mu-g} . \]
Assume that $E$ is stable.
Let $E' \subsetneq E$ be a globally generated proper subsheaf of $E$.
In \hyref{Corollary}{cor:CT1} (\hyref{Section}{sec:clifford} is independent of the rest of the article), we draw the following consequence from the generalised Clifford theorem:
 $h^0(E')-\rank(E') < \frac{\mu -g}{\mu} \deg(E')$.
Thus
\[    \mu(\Fm^0(E')) = \frac{-\deg(E')}{ h^0(E')-\rank(E')}
   <  \frac{-\mu}{\mu-g } = \mu(\Fm^0(E)) \, . \]
By \hyref{Lemma}{lem:good1}, to show the stability of $\Fm^0(E)$ it suffices to check this inequality for the subsheaves of type $\Fm^0(E')$.
The claim about semistability follows from this by the Jordan--H\"older filtration of $E$.
\end{proof}

\begin{corollary} \label{cor:STAB1}
Let $E$ be a (semi)stable vector bundle of slope $\mu(E) > 2g-1$.
Then the restriction of $\shat{E}^0$ to any curve $(-X)_M$ is (semi)stable.
In particular, $\shat{E}^0$ is (semi)stable.
\end{corollary}

\begin{proof} 
We have $\shat{E}^0|_{(-X)_M} \cong \Fm^0(E\otimes M) \otimes \ko_X(P_0)$ from \hyref{Proposition}{prop:FM-restriction}. This holds for all curves $(-X)_M$ because $E$ semistable of slope $>2g-1$ implies that $\shat{E}^0$ is locally free.
Moreover, $E\otimes M$ is a (semi)stable bundle of slope $\mu(E\otimes M) > 2g$, so \hyref{Corollary}{cor:general1} applies and yields the (semi)stability of $\Fm^0(E\otimes M)$ and hence of $\shat{E}^0|_{(-X)_M}$.
\end{proof}

\section{Application of orthogonality}

\noindent
There are two ways how to apply the orthogonality condition $E \ndorth F$, in order to deduce the semistability of the Picard bundle $\shat{F}$ from the semistability of another Picard sheaf $\shat{F}$. Either we use the symmetry of orthogonality, i.e.\ \hyref{Lemma}{lem:sym}(i) and (ii), or  we employ Serre duality, i.e.\ \hyref{Lemma}{lem:sym}(iii).
We start with the latter method.

\begin{corollary} \label{cor:res2}
Let $E$ be a semistable vector bundle on $X$ with $\mu(E)=2g-1$.
Then $\shat{E}^1 = 0$ and $\shat{E}^0$ is a vector bundle of rank $g \cdot \rank(E)$ which is also semistable. 
Moreover, the restriction of $\shat{E}^0$ to any curve $(X)_N$ is semistable.
\end{corollary}

\begin{proof}
The vanishing of $\shat{E}^1$ follows from cohomology and base change, and for the same reason, $\shat{E}^0$ is a vector bundle of the given rank.
As $E$ is semistable, so is its dual $E \dual$. Thus, $E\dual \otimes \omega_X$ is semistable of degree $-1$. By \hyref{Lemma}{lem:orth-pair}, there exists a sheaf $F$ such that $(E\dual \otimes \omega_X \otimes M) \ndorth F$
for any line bundle $M$ of degree 1. By Serre duality,  \hyref{Lemma}{lem:sym}(iii), this implies
 $(E \otimes M\dual) \dorth (F\dual \otimes \omega_X)$.
Now we proceed as in the proof of \hyref{Corollary}{cor:res1}.
\end{proof}

\begin{remark}\label{rem:general1a}
Applying \hyref{Corollary}{cor:res2} to a degree $2g-1$ line bundle $L$, we obtain the semistability of the Picard bundle $P_{2g-1} = \shat{L}^0$. Thus, the above result is a generalisation of Kempf's result \cite{Kempf}.
In fact, Kempf shows the stability of $P_{2g-1}$. The stability follows along the lines of Corollaries~\ref{cor:general1} and \ref{cor:CT1}.
Indeed, if $X$ is not hyperelliptic, and $L \otimes M\otimes \omega_X\dual$ is not effective, then the restriction of $\shat{L}^0$ to $(-X)_M$ is stable.
\end{remark}

\begin{lemma}\label{lem:general2}
If $E$ is a stable vector bundle on $X$ with $\mu(E)<-2$, then $\Fp^1(E)$ is stable.
\end{lemma}

\begin{proof}
Tensoring the short exact sequence
 $0 \to \ko_{X \times X} \to \ko_{X\times X}(\Delta) \to \ko_\Delta(\Delta) \to 0$
on $X \times X$ with $E\boxtimes\ko_X$, and applying $\pr_{2*}$ yields the following short exact sequence on $X$:
\[ 0 \to E \otimes \omega\dual \to H^1(E) \otimes  \ko_X \to \Fp^1(E) \to 0 ,\]
because $\Fm^0(E)=0$ from stability of $E$ with $\mu(E)<-2$.
Dualising this sequence yields
\[ 0 \to \left( \Fp^1(E) \right)\dual \to H^1(E) \dual \otimes  \ko_X \to E \dual \otimes \omega_X \to 0 . \]
Thus, by classical Serre duality, $\Fp^1(E)$ is the dual of $\Fm^0(E')$ for
$E'=E\dual \otimes \omega_X$. However $E'$ is also stable, of slope
 $\mu(E') = 2g-2 -\mu(E) > 2g$.
So by \hyref{Corollary}{cor:general1} the sheaf $\Fm^0(E')$ is stable.
This proves the lemma.
\end{proof}

\begin{corollary}\label{cor:STAB2}
If the vector bundle $E$ on $X$ is (semi)stable of slope $\mu(E)< -1$, then
the Picard sheaf $\shat{E}^1$ is (semi)stable when restricted to any curve $(X)_N$.
In particular, $\shat{E}^1$ is (semi)stable.
\end{corollary}

\begin{proof}
For stability, this follows from \hyref{Proposition}{prop:FM-restriction} and \hyref{Lemma}{lem:general2}.

The claim for semistable $E$ then follows using the Jordan--H\"older filtration of $E$.
\end{proof}

Next, we give examples for how to apply the symmetry property of orthogonality.
As usual, for a rational number $x$ we denote by $\lceil x\rceil$ the round up of $x$.
For any $r,h\in\IN$, we introduce the number used by Popa in \cite[Theorem 5.3]{Popa} for an effective version of \hyref{Faltings' Theorem}{thm:Faltings}:
\[ P(r,h) \coloneqq 2h \left\lceil\frac{h^2r^2+1}{8h}\right\rceil \,. \]
%and it is used in the following \cite[Theorem 5.3]{Popa}:

\begin{theorem}[Popa 2001] \label{thm:Popa}
For any semistable vector bundle $G$ of rank $r \cdot h$ and degree $d \cdot h$ with
coprime integers $r$ and $d$, there exists for any $k \geq P(r,h)$ a vector
bundle $F$ of rank $r \cdot k$ such that $H^*(G \otimes F)=0$.
\end{theorem}

\begin{lemma}\label{lem:general3}
Let $\mu=\frac{d}{r} \in \IQ$ with $\mu\in (g,g+1]$, let $k \geq P(r,g)$ and $R = k \cdot r$. 
Then there exists a vector bundle $F$ on $X$ of slope $\mu$ and rank $R$ such that $\Fm^0(F) \oplus \Fm^1(F)$ is semistable.
\end{lemma}

\begin{proof}
We begin with the involution $\IQ\to\IQ, \mu \mapsto \mu^- \coloneqq g+\smash{\frac{g}{\mu-g}}$. It is decreasing on $\IQ_{>g}$.
By \hyref{Lemma}{lem:numORTH}, if $E \ndorth F$, then $\mu(E)=\mu(F)^-$.

Now let $\mu = \frac{d}{r} \in (g,g+1]$, then $\mu^- \in [2g, \infty)$.
%  We have $\mu^- = \frac{gd-g^2r+rg}{d-rg} \geq 2g$.
Let $E$ be a stable vector bundle on $X$ of rank $d-rg$ and degree $gd-g^2r+rg$, i.e.\
$\mu(E)=\mu^-$. By \hyref{Corollary}{cor:general1} or \hyref{Remark}{rem:general1a}, we have that
$\Fm^0(E)$ is semistable. Since $\mu(E) >2g-1$, we also conclude $\Fm^1(E)=0$.
So we can use the Riemann--Roch formula to compute $\rank(\Fm^0(E))=gr$, and
$\deg(\Fm^0(E))=gr-gd-g^2r$. Popa's result Theorem \ref{thm:Popa} implies that for
any $R=k \cdot r$ with $k$ as above there exists a vector bundle $F$ of rank $R$ such
that $H^*(\Fm^0(E) \otimes F)=0$.

By \hyref{Lemma}{lem:ortho1}, this yields $E \ndorth F$.
Symmetry, i.e.\ \hyref{Lemma}{lem:sym}(i), then gives $F \ndorth  E$.
And so, again by \hyref{Lemma}{lem:ortho1}, it follows that
 $H^*((\Fm^0(F) \oplus \Fm^1(F)) \otimes E)=0$.
This implies the semistability of the direct sum $\Fm^0(F) \oplus \Fm^1(F)$.
\end{proof}

\begin{proposition}\label{prop:STAB3}
Let $\mu=\frac{d}{r} \in \IQ$ with $\mu\in(g-1,g]$, let $k \geq P(r,g)$ and $R=k \cdot r$.
Then there exists a vector bundle $F$ on $X$ of slope $\mu$ and rank $R$ such that $\shat{F}^0$ is semistable. Indeed, the restriction of $\shat{F}^0$ to the general curve $(-X)_M$ is semistable.
\end{proposition}

\begin{proof}
It is enough to show the existence of some $F$ such that $\shat{F}^0$ restricted to $(-X)_M$ is semistable. Let $\mu$ and $R$ be as in the
proposition. Let $F'$ be a vector bundle of slope $\mu+1$ and rank $R$ such that $\Fm^0(F) \oplus \Fm^1(F)$ is semistable, which exists by \hyref{Lemma}{lem:general3}. Then $\Fm^0(F')  \otimes \ko_X(P_0)$ is also semistable.

We set $F=F' \otimes M\dual$ for some line bundle $M$ of degree 1. By \hyref{Proposition}{prop:FM-restriction}, the restriction of $\shat{F}^0$ to $(-X)_M$ is the semistable sheaf $\Fm^0(F') \otimes \ko_X(P_0)$.
\end{proof}

\begin{corollary}\label{cor:STAB4}
Let $\mu=\frac{d}{r} \in \IQ$ with $\mu \in [g-2,g-1)$, let $k \geq P(r,g)$ and $R=k \cdot r$. 
Then there exists a vector bundle $F$ on $X$ of slope $\mu$ and rank $R$ such that $\shat{F}^1$ is semistable. Indeed, the restriction of $\shat{F}^1$ to the general curve $(X)_N$ is semistable.
\end{corollary}

\begin{proof}
This follows from \hyref{Proposition}{prop:STAB3}, using Serre duality as in \hyref{Lemma}{lem:sym}(iii).
\end{proof}

\section{Clifford's theorem for vector bundles on a curve} \label{sec:clifford}
\noindent
Let us remind the reader that $\mu_{\max }(E)$ denotes the maximum of all
slopes of subbundles of $E$, and $\mu_{\min}(E)$ denotes the minimal slope of
all quotient bundles of $E$.

\begin{proposition} \label{Clifford}
Let $E$ be a vector bundle of rank $r$ and degree $d$ on the smooth projective curve $X$ of genus $g$.
If $\mu_{\max}(E) \leq 2g-2 $ and $\mu_{\min}(E) \geq 0$, then we have   the estimate
\[ h^0(E)-r \leq \frac{d}{2} \, .\]
Moreover, if $\mu_{\max}(E) < 2g-2 $ and $\mu_{\min}(E) > 0$ and $h^0(E)-r = \frac{d}{2}$, then $X$ is hyperelliptic, and the determinant line bundle $\det(E)$ is a multiple of the hyperelliptic line bundle $M$, and $E$ possesses a filtration with graded object ${\rm gr}(E)= \bigoplus_{i=1}^r M^{ \otimes a_i}$ with $0<a_i<g-1$.
\end{proposition}

\begin{proof}
We first prove the inequality $h^0(E)-r \leq \frac{d}{2}$ by induction on $r$.

If $r=1$, then $E=\ko(D)$ is a line bundle associated to a divisor $D$. In this case $d=\mu_{\max}(\ko(D))=\mu_{\min}(\ko(D))$. If $D$ is effective and special, then the claim is precisely the well known theorem of Clifford, see for example \cite[Theorem IV.5.4]{Hartshorne}. If $D$ is not effective, then $h^0(\ko(D))=0$ and the statement is trivial. If $D$ is non-special, then by Riemann--Roch
 $h^0(\ko(D))-1 = \chi(\ko(D))-1 = d-g < \frac{d}{2}$,
the inequality following from $d = \mu_{\max} \leq 2g-2$.

Now suppose that $E$ is of rank $r\geq2$, and the inequality holds for all vector bundles of rank smaller than $r$ which meet the slope conditions. We consider two cases:
\\
Case 1: $E$ is not semistable. Take the subsheaf $E_1$ of $E$ of slope $\mu_{\max}(E)$  and of maximal possible rank.
This $E_1$ is the first sheaf appearing in the Harder--Narasimhan filtration of $E$. We obtain a short exact sequence
\[ 0 \to E_1 \to E \to E_2 \to 0 \, .\]
We have
 $\mu_{\max}(E_1)=\mu_{\min}(E_1)=\mu_{\max}(E)$, $\mu_{\max}(E) > \mu_{\max}(E_2)$, and $\mu_{\min}(E)=\mu_{\min}(E_2)$.
In particular, we see that the induction hypothesis applies to the vector bundles $E_1$ and $E_2$. Hence
 $h^0(E_i) -\rank(E_i) \leq \frac{1}{2}\deg(E_i)$
for $i\in \{ 1,2\}$. Taking global sections of the above short exact sequence, we conclude
 $h^0(E) \leq h^0(E_1)+h^0(E_2)$.
So we get
\[ \begin{array}{rclcl}
        h^0(E)-r & \leq &  h^0(E_1)+h^0(E_2)-r
& = &   \left( h^0(E_1) -\rank(E_1) \right) + (h^0(E_2) -\rank(E_2)) \\
&\leq&  \frac{1}{2} \deg(E_1) + \frac{1}{2} \deg(E_2) &=& \frac{1}{2} \deg(E) .
\end{array} \]
Case 2: $E$ is semistable.
Again, we distinguish two cases, by inspecting the slope of $E$.
\\
Case 2.1: $\mu(E) \leq g-1$.
We may assume $h^0(E)>0$. Let $E_1$ be a line subbundle of $E$ of maximal possible degree $d_1$. From $h^0(E)>0$ we conclude that $d_1 \geq 0$. We obtain a short exact sequence
\[ 0 \to E_1 \to E \to E_2 \to 0 \, .\]
Since any quotient of $E_2$ is also a quotient of $E$ we conclude $\mu_{\min}(E_2) \geq \mu_{\min}(E) \geq 0$. We want to show that $\mu_{\max}(E_2) \leq 2g-2$.
Assume the contrary. Then we have a subsheaf $E_3 \subset E_2$ of rank $r_3$ and slope $\mu(E_3) > 2g-2$. The kernel $K$ of the composition of surjections
\[ E \to E_2 \to E_2/E_3 \]
is of rank $r_3+1$ and of slope
 $\mu(K) = \frac{d_1+\mu(E_3)\cdot r_3}{r_3+1}
      \geq \frac{\mu(E_3)\cdot r_3}{r_3+1}
         = \frac{\mu(E_3)}{1+1/r_3}
         > \frac{2g-2    }{1+1/r_3}
      \geq g-1$.
This contradicts the semistability of $E$. Thus, for both sheaves $E_1$ and $E_2$ the induction hypothesis applies, and we can proceed like in Case 1.
\\
Case 2.2: $\mu(E_1) > g-1$.
The Serre dual bundle $E'=E\dual \otimes \omega_X$ has slope
 $\mu(E') = 2g-2 -\mu(E) < g-1$.
Therefore, as we have seen in Case 2.1
\[ h^0(E') \leq \frac{\deg(E')}{2} = \frac{\rank(E)(2g-2)-\deg(E)}{2} \, .\]
By Serre duality $h^0(E') = h^1(E)$. So when adding the Riemann--Roch formula
 $h^0(E)-h^1(E) = \deg(E)+\rank(E)(1-g)$
to the above inequality, we obtain the stated inequality.

The statement for the case $h^0(E)-r = \frac{d}{2}$ follows along the same lines. Indeed, we must have this equality for $E_1$ and $E_2$ and can proceed by induction since
 $\det(E) \cong \det(E_1) \otimes \det(E_2)$.
The passage from $E$ to the Serre dual $E'=E\dual \otimes \omega_X$ sends a vector bundle $E$ with $\det(E)=M^{\otimes a}$ to a bundle with 
 $\det(E')=M^{\otimes(r(g-1)-a)}$
where $M$ denotes the hyperelliptic line bundle.
\end{proof}

\begin{corollary}\label{cor:CT1}
Let $E$ be a stable vector bundle of slope $\mu = \mu(E) > 2g$.
For any globally generated subsheaf $E' \subsetneq E$ which is not a trivial bundle, % such that $h^0(E')-\rank(E')>0$, 
we have the strict inequality
\[  h^0(E')-\rank(E') < \Big( 1 - \frac{g}{\mu}\Big) \deg(E') . \]
\end{corollary}

\begin{proof}
As $E' \subset E$ is globally generated, we have $\mu_{\min}(E') \geq 0$.
For one sheaf $E_1'$ in the Harder--Narasimhan filtration of $E'$, we have $\mu_{\min}(E_1')>2g-2$, and $\mu_{\max}(E'/E_1') \leq 2g-2$.
We set $E_2'=E'/E_1'$. Now $E_1'$ is semistable with $\mu_{\min}(E_1')>2g-2$, hence $h^1(E_1')=0$. So
\[ h^0(E_1')-\rank(E_1') = \chi(E_1') - \rank(E_1')
                         = \deg(E_1') - g\cdot\rank(E_1')
                         = \Big( 1 - \frac{g}{\mu(E_1')} \Big) \deg(E_1') \]
by Riemann--Roch. Since $\mu(E_1') < \mu$, and the function $x \mapsto 1-\frac{g}{x}$ is strictly increasing for $x>0$, we deduce the inequality
\[ h^0(E_1')-\rank(E_1') < \Big( 1- \frac{g}{\mu} \Big) \deg(E_1') . \]
The sheaf $E_2'$ satisfies the assumptions of \hyref{Proposition}{Clifford}. Moreover, $E'_2$ is itself a globally generated sheaf, and we have
 $\mu_{\min}(E_2') \geq \mu_{\min}(E') > 0$,
the latter inequality from the assumption that $E'$ is not a trivial bundle. Hence $\deg(E_2')>0$ and so we have 
\[ h^0(E_2')-\rank(E_2') \leq \frac{1}{2} \deg(E_2') < \Big( 1 - \frac{g}{\mu} \Big) \deg(E_2') . \]
This last inequality holds, because $\mu > 2g$ implies $\frac{1}{2} < 1-g/\mu$. Adding the two inequalities for $h^0(E_i')-\rank(E_i')$ for $i =1,2$, we obtain the statement of the corollary.
\end{proof}

\bigskip
\noindent
Email:   \texttt{georg.hein@uni-due.de, ploog@math.uni-hannover.de}
\end{document}